\documentclass[a4paper]{article}
\usepackage{amsmath,mathrsfs,amsfonts,amsthm,amssymb,latexsym}
\usepackage{enumerate}
\usepackage{color}
\usepackage{hyperref}

\newtheorem{theorem}{Theorem}[section]
\newtheorem{lemma}[theorem]{Lemma}
\newtheorem{corollary}[theorem]{Corollary}
\theoremstyle{definition}
\newtheorem{definition}[theorem]{Definition}
\newtheorem{proposition}[theorem]{Proposition}

\newtheorem{example}[theorem]{Example}
\newtheorem{remark}[theorem]{Remark}

\DeclareMathOperator{\SQS}{SQS}
\DeclareMathOperator{\AG}{AG}
\DeclareMathOperator{\GL}{GL}
\DeclareMathOperator{\AGL}{AGL}
\DeclareMathOperator{\PSL}{PSL}
\DeclareMathOperator{\Stab}{Stab}

\newcommand{\bbox}[1]{\fbox{$\boldsymbol{#1}$}}
\newcommand{\vct}[1]{\boldsymbol{#1}}

\newcommand{\gf}[1]{\mathbb{F}_{#1}}

\newcommand{\bbZ}[0]{\mathbb{Z}}

\newcommand{\calB}[0]{\mathcal{B}}
\newcommand{\calC}[0]{\mathcal{C}}
\newcommand{\calD}[0]{\mathcal{D}}

\newcommand{\calP}[0]{\mathcal{P}}
\newcommand{\calQ}[0]{\mathcal{Q}}

\newcommand{\scrB}[0]{\mathscr{B}}
\newcommand{\scrC}[0]{\mathscr{C}}
\newcommand{\scrO}[0]{\mathscr{O}}

\newcommand{\abs}[1]{\left|{#1}\right|}
\newcommand{\gen}[1]{\left\langle {#1} \right\rangle}

\title{Completely (Quasi-)Uniform Nested Boolean Steiner Quadruple Systems}

\author{{\sc Xiao-Nan Lu}\footnote{Department of Electrical, Electronics and Computer Engineering, Gifu University, 1-1 Yanagido, Gifu, 501-1193, Japan. Email: \texttt{lu.xiaonan.b3@f.gifu-u.ac.jp}.}}
\date{} 

\begin{document}
\maketitle

\begin{abstract}
Nested Steiner quadruple systems are designs derived from Steiner quadruple systems (SQSs)  by partitioning each block into pairs. 
A nested SQS is completely uniform if every possible pair appears with equal multiplicity, and completely quasi-uniform if every pair appears with multiplicities that differ by at most one.
An explicit construction on the Boolean SQS of order $2^m$ is presented, producing a nested SQS$(2^m)$ that is completely uniform when $m$ is odd and completely quasi-uniform when $m$ is even for each integer $m \ge 3$.
These results resolve two open problems posed by Chee et al.~(2025). 
The notion of completely uniform pairings is further generalized for $t$-designs with $t \ge 2$. 
As an application, completely uniform nested $2$-$(2^m,4,3)$ designs give rise to fractional repetition codes with zero skip cost, requiring fewer storage nodes than constructions based on SQSs.
In addition, small examples are provided for non-Boolean orders, establishing the existence of completely uniform nested SQS$(v)$ for all $v \le 50$.
\end{abstract} 



\section{Introduction}

A $t$-$(v,k,\lambda)$ design is a pair $(V, \calB)$, where $V$ is a set of $v$ points and $\calB$ is a collection of $k$-subsets of $V$, called blocks, such that every $t$-subset of $V$ is contained in exactly $\lambda$ blocks in $\calB$. 
A $3$-$(v, 4, 1)$ design is called a \emph{Steiner quadruple system} (SQS) of order $v$, denoted by $\SQS(v)$. Hanani~\cite{Hanani1960} proved in 1960 that an $\SQS(v)$ exists if and only if $v \equiv 2$ or $4 \pmod{6}$. 
Since then, considerable attention has been devoted to constructing SQSs with prescribed symmetry and resolvability.
In particular, the study of SQSs admitting cyclic groups, dihedral groups, and more generally abelian groups acting regularly on the points was pioneered by K\"ohler~\cite{Kohler1978} and Piotrowski~\cite{Piotrowski1985}.
Recently, Ji and Lu~\cite{JiLu2021} showed that the existence problem for symmetric abelian group-invariant SQSs reduces to the case of symmetric cyclic $\SQS(2p)$, where $p$ is an admissible prime. However, the existence of such cyclic SQSs remains unresolved in general (see \cite{JiLu2021} and the references therein).
At the opposite extreme of symmetry, Huber~\cite{Huber2001} classified all flag-transitive SQSs into three classes, one of which contains the Boolean SQSs of order $2^m$ with $m \geq 3$, the main object of this work.
Resolvability is another central theme. 
Hartman~\cite{hartman1987existence} proved that a resolvable $\SQS(v)$ exists for all admissible orders $v \equiv 4,8 \pmod{12}$ except for 23 possible exceptions, and the remaining cases were finally established by Ji and Zhu~\cite{JiZhu2005}.  Nevertheless, many fundamental problems concerning SQSs, and more generally Steiner $3$-designs, remain largely open. 
The relevant definitions and necessary background will be provided in Section~\ref{sect:preliminaries}. 

Motivated by applications to fractional repetition (FR) codes with zero skip cost~\cite{chee2024repairing}, the concept of \emph{nested Steiner quadruple systems} (nested SQSs) was recently introduced by Chee et al.~\cite{chee2024pairs}. The definitions and examples related to FR codes will be presented in Section~\ref{sect:applications}. 
A \emph{nested SQS} is derived from an SQS by partitioning each block of the SQS into two pairs.  
A formal definition can be given as follows: 

\begin{definition}[see \cite{chee2024pairs}]
Let $V$ be a set of $v$ points. 
A \emph{nested Steiner quadruple system} (nested SQS) is a pair $(V, \scrB)$, where $\scrB$ is a subset of 
\[
\binom{\binom{V}{2}}{2} := \left\{ \big\{\{x, y\}, \{z, w\} \big\} :  \{x, y\}, \{z, w\} \in \binom{V}{2} \right\},
\] 
such that $(V, \calB)$ is an SQS with 
$\calB = \big\{ \{x, y, z, w\} :  \big\{\{x, y\}, \{z, w\} \big\} \in \scrB \big\}$.
\end{definition}

Given a block $B = \{x,y,z,w\}\in\calB$ of an SQS, a partition of $B$ into two pairs $\widetilde{B} = \big\{\{x, y\}, \{z, w\} \big\} \in \scrB$ is called a \emph{nested block}, and the pairs $\{x, y\}$ and $\{z, w\}$ are called \emph{nested pairs}.
In this paper, we simply write $\{x, y \mid z, w\}$ to denote the nested block $\big\{\{x, y\}, \{z, w\} \big\}$. 

The number of times a given pair appears in $\scrB$ is called its \emph{multiplicity}.
Bounds on multiplicities are extensively investigated in~\cite{chee2024pairs}, where several recursive constructions (extending the well-known doubling constructions for SQSs) are also provided.

A nested SQS is said to be \emph{uniform} if all nested pairs occur with the same multiplicity. 
Note that, in a uniform nested SQS, it is not necessarily the case that every pair in $\binom{V}{2}$ appears as a nested pair. 
A uniform nested SQS is called \emph{completely uniform} if, in addition, all pairs in $\binom{V}{2}$ appear as nested pairs. 
It is known that a completely uniform nested $\SQS(v)$ can exist only when $v \equiv 2 \pmod{6}$ (see Theorem~\ref{thm:all_uniform}).  
An example of a completely uniform nested $\SQS(8)$ is given in Example~\ref{ex:8}.

Because of the restrictive necessary condition for completely uniform nested SQSs, a relaxation is considered. 
A \emph{quasi-uniform} nested $\SQS(v)$ is defined to be a non-uniform nested $\SQS(v)$ in which the difference between the multiplicities of any two nested pairs is at most one. 
A quasi-uniform nested $\SQS(v)$ is said to be \emph{completely quasi-uniform} if, moreover, every pair in $\binom{V}{2}$ appears as a nested pair. 
A necessary condition for the existence of completely quasi-uniform nested SQSs is established in Theorem~\ref{thm:all_quasi_uniform}.  
An example of a completely quasi-uniform nested $\SQS(16)$ is given in Example~\ref{ex:16}.

Chee et al.~\cite{chee2024pairs} posed several open questions on nested SQSs, among which the following two are of particular importance (rephrased Problems 1 and 3 in \cite[Section 5]{chee2024pairs}):

\begin{itemize}
\item Construct a completely uniform nested $\SQS(2^m)$ for odd integers $m \ge 3$.
\item Find an infinite family of quasi-uniform nested SQSs.
\end{itemize}

This paper gives explicit constructions that resolve both problems. 
The main result (Theorem~\ref{thm:main_boolean}) shows that, for every integer $m\ge 3$, a nested $\SQS(2^m)$ can be constructed from the Boolean $\SQS(2^m)$, and it is completely uniform when $m$ is odd and completely quasi-uniform when $m$ is even.

Beyond resolving these problems, the paper concentrates on Boolean SQSs for the following reasons.

From the viewpoint of design theory, direct constructions for SQSs (and, more generally, for $t$-designs with $t \ge 3$) are rather limited.
Among known families, Boolean SQSs are classical and have rich algebraic and finite-geometric structure, so they provide a natural entry point to completely uniform pairings, for which no direct algebraic or finite-geometric constructions were previously known. 
This study can be regarded as a first step in this direction and may help address other open problems raised in Chee et al.~\cite{chee2024pairs}.

From the viewpoint of applications, completely uniform nested SQSs are relevant to FR codes for distributed storage.
In this context, designs whose point set is a finite field of characteristic $2$, such as Boolean SQSs and their subdesigns, are especially natural, since their Boolean structure aligns with the binary representation underlying conventional computer storage. 
Thus, the Boolean cases are not only mathematically interesting but also practically important. 

The remainder of the paper is organized as follows. 
Section~\ref{sect:preliminaries} presents the necessary preliminaries 
and introduces an extension of the notion of completely uniform pairing to general $t$-designs, together with a brief summary of its relations to some known concepts in the literature. 
Section~\ref{sect:bool_construction} establishes that for every $m \ge 3$ there exist $2$-$(2^m,4,3)$ designs admitting completely uniform pairing, and this construction leads to the main theorem on nested SQSs. Section~\ref{sect:applications} shows that these nested $2$-designs yield FR codes with zero skip cost while requiring fewer storage nodes than SQS-based schemes.  Further small non-Boolean examples are given in Section~\ref{sect:more_res}, and concluding remarks are presented in Section~\ref{sect:conclude}.

\section{Preliminaries}
\label{sect:preliminaries}

\subsection{Basic concepts and notation}

For clarity, some necessary concepts from design theory are reviewed. 
A general overview of design theory is given in~\cite{colbourn2006handbook}, 
and further details on Steiner quadruple systems (SQSs) can be found in~\cite{hartman1992steiner,lindner1978steiner}.

Let $(V, \calB)$ be a $t$-$(v, k, \lambda)$ design. 
If there exists a permutation $\sigma$ on $V$ of order $v$ preserving the block set $\calB$, then $(V, \calB)$ is said to be \emph{cyclic}.
 It is convenient to use $\bbZ_v$ as the point set of a cyclic design, where $\bbZ_{v}$ denotes (the cyclic group of) the ring of integers modulo $v$. 
Similarly, the design $(V, \calB)$ is said to be \emph{$1$-rotational}, or simply \emph{rotational}, if it admits an automorphism $\sigma$ that fixes one point (say $\infty$) and acts regularly (i.e., sharply transitively) on the remaining $v-1$ points.
For rotational designs, it is natural to identify $V$ with $\bbZ_{v-1} \cup \{\infty\}$, where $x + \infty$ is defined to be $\infty$ for any $x \in \bbZ_{v-1}$.

A design $(V, \calB)$ is said to be \emph{resolvable} if there exists a partition of $\calB$ into \emph{parallel classes}, each of which is a partition of $V$. The partition of parallel classes is called a \emph{resolution}. 

In the following discussion, the point set $V$ is usually taken to be $\gf{q}$, $\bbZ_n$, or $\bbZ_n \cup \{\infty\}$, where 
$\gf{q}$ denotes the finite field of order $q$, and 
$\bbZ_n$ denotes the ring of integers modulo $n$.
In these cases, the arithmetic operations are defined with respect to the underlying algebraic structure.  
Moreover, for a block $B=\{x,y,z,w\}$ or a nested block $\widetilde{B}=\{x,y \mid z,w\}$,  
the following convenient notation  is used: 
\begin{align*}
a \cdot B + t &:= \{ ax+t,\, ay+t,\, az+t, aw+t \,\}, \\
a \cdot \widetilde{B} + t &:= \{ ax+t,\, ay+t \mid az+t, aw+t \,\}.
\end{align*}

\subsection{Boolean SQSs}
\label{sect:Prelim_Boolean_SQS}

Let $\gf{2}$ denote the finite field of order $2$, and $\gf{2}^m$ denotes the $m$-dimensional vector space over $\gf{2}$. 
The pair $(\gf{2}^m, \calB)$ is called a \emph{Boolean SQS} of order $2^m$, where 
\[
\calB = \big\{ \{ x, y, z, x+y+ z \} : 
x, y, z \in \gf{2}^m \big\}.
\]
Let $\gf{2^m}$ denote the finite field of order $2^m$ with primitive element $\alpha$. 
The space $\gf{2}^m$ can be identified with $\gf{2^m}$ by fixing a vector space isomorphism over $\gf{2}$.
The following proposition gives a convenient description for analyzing the Boolean $\SQS(2^m)$.
Since the statement follows directly from the definition, the proof is omitted.

\begin{proposition}\label{prop:B0B1Q0Q1}
The block set $\calB$ of the Boolean $\SQS(2^m)$ on the point set $\gf{2^m}$ can be partitioned into two parts as follows: 
\begin{align*}
\calB_{0} &:=
\big\{ \{ 0, \alpha^i, \alpha^j,\alpha^i+ \alpha^j \} : 
 0 \le i <  j < 2^m-1 \big\}, \\ 
\calB_{1} &:=
\big\{ \{ \alpha^i, \alpha^j, \alpha^k,\alpha^i+ \alpha^j+ \alpha^k \} : 
  0 \le  i < j < k< 2^m-1 \big\}. 
\end{align*}
Let $V = \bbZ_{2^m-1} \cup \{ \infty \}$ and define
\begin{align*}
\calQ_0 &:=
\big\{ \{ \infty, i, j, l \} : i, j, l \in \bbZ_{2^m-1}, \alpha^i+ \alpha^j = \alpha^l \big\}, \\
\calQ_1&:=
\big\{ \{ i, j, k,l \} : i, j, k, l \in \bbZ_{2^m-1},  \alpha^i+ \alpha^j+ \alpha^k = \alpha^l \big\}.
\end{align*}
Then, the design $(V, \calQ_0 \cup \calQ_1)$, which is rotational, is isomorphic to 
$(\gf{2^m}, \calB_0  \cup \calB_1)$ 
via the natural correspondence induced by the discrete logarithm base $\alpha$. 
\end{proposition}

The Boolean SQS $(\gf{2^m}, \calB)$ can be also regarded as $\AG_2(m, 2)$, the point-plane incidence structure of the $m$-dimensional affine geometry over $\gf{2}$, hence it is obviously resolvable. 
The construction of parallel classes via translation of planes in $\AG_2(m, 2)$ is formally described in the following proposition.
The proof is straightforward and hence omitted. 

\begin{proposition}\label{prop:bsqs_pc}
Given a block $B=\{ x, y, z, x + y + z \}$ in $(\gf{2^m}, \calB)$, let 
\[
\calP(B) = \big\{ B+t : t \in \gf{2^m} \big\}.
\]
Then $\calP(B)$ forms a parallel class, which consists of $2^{m-2}$ blocks. 
Moreover, the family of all parallel classes $\big\{ \calP(B) : B \in \calB\big\}$ forms a resolution of $(\gf{2^m}, \calB)$ consisting of $(2^m-1)(2^{m-1}-1)/3$ parallel classes.
\end{proposition}

\subsection{Necessary conditions for the existence of completely (quasi-)uniform nested SQSs}
\label{sect:comp_quasi_uniform}

Recall that a nested SQS is said to be \emph{uniform} if all nested pairs occur with the same multiplicity. 
It is called \emph{quasi-uniform} if the multiplicities of any two nested pairs differ by at most one, while not all nested pairs have the same multiplicity. 
Furthermore, a (quasi-)uniform nested SQS is said to be \emph{completely} (quasi-)uniform if every pair in $\binom{V}{2}$ appears as a nested pair.

\begin{example}\label{ex:8}
Let $V = \bbZ_{7} \cup \{\infty\}$ and let
\[
\widetilde{B}_0 = \{ \infty, 0 \mid 1, 3\}, \quad
\widetilde{B}_1 = \{ 2, 6 \mid 4, 5\}.
\]
Define 
\[
\scrB = \bigcup_{i=0}^{1} \left\{ \widetilde{B}_i + t : t \in \bbZ_7 \right\}.
\]
Then $(V, \scrB)$ is a completely uniform nested $\SQS(8)$ in which each pair occurs exactly once.
This nested $\SQS(8)$ is obtained from a rotational SQS, which is isomorphic to the Boolean $\SQS(8)$.
\end{example}

\begin{example}\label{ex:16}
Let $V = \bbZ_{15} \cup \{\infty\}$ and define eight base nested blocks $\widetilde{B}_0, \widetilde{B}_1, \ldots, \widetilde{B}_7$ as follows:
\begin{align*}
&\{\infty, 0 \mid 1, 4 \},  &
&\{ 2, 8 \mid 5, 10 \},  &
&\{ 3, 14 \mid 9, 7 \},  &
&\{ 6, 13 \mid 11, 12 \}, \\
&\{\infty, 0 \mid 2, 8 \}, &  
&\{ 1, 4 \mid 5, 10 \},  &
&\{ 3, 14 \mid 6, 13 \},  &
&\{ 7, 9 \mid 12, 11 \}.
\end{align*}
Define
\[
\scrB = \bigcup_{i=0}^{7} \left\{\widetilde{B}_i + t : t \in \bbZ_{15}\right\}.
\]
Moreover, define four additional base nested blocks $\widetilde{C}_0, \widetilde{C}_1, \widetilde{C}_2, \widetilde{C}_3$ as follows:
\begin{align*}
&\{\infty, 0 \mid 5, 10 \}, &
&\{ 1, 4 \mid 2, 8 \}, &
&\{ 3, 14 \mid 11, 12 \}, &
&\{ 6, 13 \mid 9, 7 \}.
\end{align*}
Let
\[
\scrC = \bigcup_{i=0}^{3} \left\{\widetilde{C}_i + t : t \in \{0,1,2,3,4\}\right\}.
\]
Then, $(V,\scrB \cup \scrC)$ is a completely quasi-uniform nested $\SQS(16)$.
Every pair in $\binom{V}{2}$ occurs exactly twice in $\scrB$, and $\scrC$ contributes $40$ additional distinct pairs once. Hence, in $\scrB \cup \scrC$ precisely $40$ pairs have multiplicity $3$ and the remaining $80$ pairs have multiplicity $2$, covering all $\binom{16}{2}=120$ pairs.
This $\SQS(16)$ is obtained from a rotational SQS that is isomorphic to the Boolean $\SQS(16)$.
\end{example}

A necessary condition for the existence of a completely uniform nested SQS is given by the following theorem.

\begin{theorem}[{\cite[Theorem 4.1]{chee2024pairs}}]
\label{thm:all_uniform}
If a completely uniform nested $\SQS(v)$ exists, then $v \equiv 2 \pmod{6}$ and the multiplicity of each pair is $\frac{v-2}{6}$.
\end{theorem}

While the concept of quasi-uniform nested SQSs was introduced in~\cite{chee2024pairs}, their constructions and further properties were not extensively examined.
As an analog of Theorem~\ref{thm:all_uniform}, 
we derive below a necessary condition for the existence of a completely quasi-uniform nested SQS.

\begin{theorem}\label{thm:all_quasi_uniform}
If a completely quasi-uniform nested $\SQS(v)$ exists, then $v \equiv 4 \pmod{6}$ and the multiplicity of each pair is either $\frac{v-4}{6}$ or $\frac{v+2}{6}$.
More precisely, in a completely quasi-uniform nested $\SQS(v)$, there are 
$\frac{v(v-1)}{3}$ nested pairs with multiplicity $\frac{v-4}{6}$
and $\frac{v(v-1)}{6}$ nested pairs with multiplicity $\frac{v+2}{6}$. 
\end{theorem}
\begin{proof}
Suppose there are $n_0$ nested pairs with multiplicity $\mu$ and $n_1$ nested pairs with multiplicity $\mu+1$.  
Then the following equations hold:
\begin{align}
n_0 + n_1 &= \binom{v}{2}, \label{eq:all_quasi_uniform_1} \\
\mu n_0 + (\mu + 1) n_1 &= \frac{v(v-1)(v-2)}{12}. \label{eq:all_quasi_uniform_2}
\end{align}
Eliminating $n_0$ and solving for $\mu$ gives
$\mu= \frac{v-2}{6} - {n_1}/{\binom{v}{2}}$. 
Since $0 < n_1 < \binom{v}{2}$, it follows that
$\frac{v-2}{6} - 1 < \mu < \frac{v-2}{6}$.
No integer $\mu$ satisfies this inequality when $v \equiv 2 \pmod{6}$.  
Since an $\SQS(v)$ exists if and only if $v \equiv 2,4 \pmod{6}$, it must be the case that $v \equiv 4 \pmod{6}$, and in this case $\mu = \tfrac{v-4}{6}$.  
Finally, solving \eqref{eq:all_quasi_uniform_1} and \eqref{eq:all_quasi_uniform_2} for $n_0$ and $n_1$ yields
$n_0 = \frac{v(v-1)}{3}$ and $n_1 = \frac{v(v-1)}{6}$. 
\end{proof}

\subsection{Generalization of completely  (quasi-)uniform nested designs}
\label{sect:gen_design}

The notions of completely (quasi-)uniform nested $\SQS$s can be generalized for $t$-$(v, k, \lambda)$ designs for any $t \ge 2$, by partitioning the elements of each block of size $k$ into subsets of size $l$, where $l$ divides $k$.

\begin{definition}\label{def:t-pairing}
Let $t \ge 2$, $v \ge k > l \ge 2$ be integers such that $l$ divides $k$. 
An \emph{$l$-nested $t$-$(v, k, \lambda)$ design} is a pair $(V, \scrB)$, where $V$ is a set of $v$ points and
$\scrB$ is a subset of 
\[
\binom{\binom{V}{l}}{k/l} := \left\{ \big\{ S_1, S_2 \ldots, S_{k/l}\big\}\;  :  \; S_1, S_2 \ldots, S_{k/l} \in \binom{V}{l} \right\},
\] 
such that $(V, \calB)$ is a $t$-$(v, k, \lambda)$ design with 
\[
\calB = \left\{ \bigcup_{i=1}^{k/l} S_i \; : \;  \big\{ S_1, S_2 \ldots, S_{k/l}\big\} \in \scrB \right\}.
\]
An $l$-nested design is called \emph{uniform} if all nested subblocks $S_i$ have the same multiplicity.
It is called \emph{quasi-uniform} if it is not uniform but the multiplicities of nested subblocks differ by at most one.
A (quasi-)uniform $l$-nested design is said to be \emph{completely} (quasi-)uniform if every subset in $\binom{V}{l}$ appears as a nested subblock.
In particular, when $l=2$, a subblock is simply called a \emph{pair}, and a $2$-nested design is simply called a \emph{nested} design.
\end{definition}

There is a closely related notion called \emph{nested balanced incomplete block designs} (nested BIBDs), originating from statistical design of experiments and first formally introduced by Preece~\cite{Preece1967}.  

\begin{definition}[cf.~\cite{MorganPreeceRees2001,Preece1967}]\label{def:preece}
Let $(V, \calD_1)$ be a $2$-design (BIBD) with $v$ points, $b_1$ blocks, and block size $k_1$.  
Suppose that each block of $\calD_1$ can be partitioned into subblocks of size $k_2$, where $k_2 \ge 2$ is a divisor of $k_1$, and that the resulting $b_2 = b_1 k_1 / k_2$ subblocks themselves form a $2$-design $(V, \calD_2)$ with block size $k_2$.  
Then, the triple $(V, \calD_1, \calD_2)$ is called a \emph{nested BIBD} with parameters $(v, b_1, b_2, k_1, k_2)$. 
To distinguish this type of nested BIBD, it will be referred to as a \emph{Preece nested BIBD}.
\end{definition}

The notion of a Preece nested BIBD can be extended to a \emph{doubly nested BIBD}, in which blocks and subblocks are defined as before, but each subblock is further partitioned into subsubblocks that themselves form a BIBD.  
This idea extends naturally to \emph{multiply nested BIBDs}, where blocks are recursively partitioned into smaller BIBDs at successive levels.  For further details and related results, see~\cite{MorganPreeceRees2001}.

Note that a $t$-$(v, k, \lambda)$ design is also a $(t-1)$-$(v, k, \lambda_{t-1})$ design, where
$\lambda_{t-1} = \lambda \cdot \frac{v-t+1}{k-t+1}$.
Hence, for $t \ge 3$ and $l \ge 2$, 
a completely uniform $l$-nested $t$-$(v, k, \lambda)$ design having $b$ blocks in Definition~\ref{def:t-pairing} is precisely a Preece nested BIBD with parameters $(v, b, kb/l, k, l)$ in Definition~\ref{def:preece}. The collection of subblocks in this case forms a trivial $l$-design with $v$ points and block size $l$.

\begin{remark}
There are other notions in combinatorial design theory with similar terminology, such as \emph{balanced nested designs} or \emph{split-block designs}~\cite{chisaki2020combinatorial,fuji2002balanced,kuriki1994balanced,ozawa2002optimality}, where blocks are partitioned into subblocks under different balancedness conditions.  
The term \emph{nest} is also used in another sense~\cite{buratti2024nestings,stinson1985spectrum}, referring to the extension of one design to another by adding ``nested'' points to blocks.  
These concepts share only similar names but are essentially distinct from the notion considered in this paper.
\end{remark}

In this paper, we focus on the cases $t \in \{2,3\}$ with block size $k = 4$, and consider partitions of blocks into pairs, i.e., $l = 2$.
A completely uniform $2$-nested $2$-$(v,4,\lambda)$ design in our Definition~\ref{def:t-pairing} is equivalent to a Preece nested BIBD  with parameters $(v, b, 2b, 4, 2)$, where $b = \tfrac{\lambda v(v-1)}{12}$.  
In particular, since an $\SQS(v)$ is also a $2$-$(v,4,\tfrac{v-2}{2})$ design,  
 every completely uniform nested $\SQS(v)$ gives rise to a Preece nested BIBD with parameters $(v, b, 2b, r, 4, 2)$, where $b=\tfrac{v(v-1)(v-2)}{24}$. However, not every Preece nested BIBD with these parameters corresponds to a completely uniform nested $\SQS(v)$.

It is worth noting that the case $k=4$, $\lambda=3$, and $\ell=2$ is also closely related to the classical problem of whist tournaments, whose study dates back to Moore's study in 1896 \cite{Moore1896a,Moore1896b}.
More precisely, a whist tournament of order $v=4n$ is exactly a completely uniform resolvable $2$-nested $2$-$(v,4,3)$ design (see \cite{AndersonFinizio-Whist-2006,MorganPreeceRees2001} for details).

\begin{lemma}
For a completely uniform nested $2$-$(v,4,\lambda)$ design to exist, it is necessary that $\lambda \equiv 0 \pmod{3}$.
\end{lemma}

\begin{proof}
In a $2$-$(v,4,\lambda)$ design, the number of blocks is
$b = \lambda \cdot \frac{v(v-1)}{12}$.
Each block contributes exactly two nested pairs, so the total number of nested pairs in such a nested design is
$2b = \lambda \cdot \frac{v(v-1)}{6}$.
Complete uniformity requires that all $\binom{v}{2}$ pairs occur an equal number of times.
Hence, $2b$ must be divisible by $\frac{v(v-1)}{2}$, equivalently, $\lambda \equiv 0 \pmod{3}$.
\end{proof}

The following propositions provide sufficient conditions for the existence of completely (quasi-)uniform nested SQSs. They show that the construction problems for SQSs can be reduced to that of sub-$2$-designs. This reduction will be useful in the proof of the main theorem concerning Boolean SQSs.

\begin{proposition}\label{prop:split_2designs_2} 
Let $v \equiv 2 \pmod{6}$, and let $(V, \calB)$ be an $\SQS(v)$. 
Suppose that the block set $\calB$ can be partitioned into $\frac{v - 2}{6}$ disjoint subsets as
\[
\calB = \bigcup_{h=1}^{\frac{v - 2}{6}} \calB_h,
\]
where each $(V, \calB_h)$ is a $2$-$(v,4,3)$ design that admits a completely uniform nested pairing. Then the SQS $(V, \calB)$ also admits a completely uniform nested pairing.
\end{proposition}
\begin{proof}
The union of the nested blocks from the $\frac{v-2}{6}$ completely uniform nested $2$-$(v,4,3)$ designs forms a nested $\SQS(v)$ in which each nested pair occurs $\frac{v-2}{6}$ times, as required in Theorem~\ref{thm:all_uniform}.
\end{proof}

\begin{proposition}\label{prop:split_2designs_4}
Let $v \equiv 4 \pmod{6}$, and let $(V, \calB)$ be an $\SQS(v)$. 
Suppose that the block set $\calB$ can be partitioned as
\[
\calB = \calB_0 \cup \bigcup_{h=1}^{\frac{v - 4}{6}} \calB_h,
\]
where $(V, \calB_0)$ forms a $2$-$(v,4,1)$ design, and for each $1 \le h \le \frac{v - 4}{6}$, $(V, \calB_h)$ is a $2$-$(v,4,3)$ design that admits a completely uniform nested pairing.
Then the SQS $(V, \calB)$ admits a completely quasi-uniform nested pairing.
\end{proposition}
\begin{proof}
As in Proposition~\ref{prop:split_2designs_2}, combining the nested blocks from the $\frac{v-4}{6}$ completely uniform nested $2$-$(v,4,3)$ designs with those from a nested $2$-$(v,4,1)$ design yields a nested $\SQS(v)$. 
Here, for the $2$-$(v,4,1)$ design, the blocks can be partitioned arbitrarily to form a nested $2$-$(v,4,1)$ design. This nested design is uniform but not completely uniform, since each pair appears only once and thus covers $\frac{v(v-1)}{6}$ nested pairs in total.  
The resulting nested $\SQS(v)$ has each nested pair occurring $\frac{v - 4}{6}$ or $\frac{v + 2}{6}$ times, as required in Theorem~\ref{thm:all_quasi_uniform}.
\end{proof}

\section{Construction for completely (quasi-)uniform nested $\SQS(2^m)$}
\label{sect:bool_construction}

Let $m \ge 3$ and consider the Boolean SQS $(\gf{2^m}, \calB)$ as defined in Section~\ref{sect:Prelim_Boolean_SQS}.

As shown in Proposition~\ref{prop:bsqs_pc}, the mapping $x \mapsto x+t$ over $\gf{2^m}$, corresponding to the action of the additive group of $\gf{2^m}$, yields parallel classes in $(\gf{2^m}, \calB)$.  
By further applying the action of the multiplicative group of $\gf{2^m}$ to these parallel classes, one obtains a $2$-$(2^m,4,\lambda)$ subdesign of $(\gf{2^m}, \calB)$ with $\lambda \in \{1,3\}$, as stated in the following lemma.  
Here, the mapping $x \mapsto \alpha^k \cdot x$ over $\gf{2^m}$ can be interpreted as $i \mapsto k+i$ on $\bbZ_{2^m-1} \cup \{\infty\}$ (with $\infty \mapsto \infty$), that is, as a cyclic translation on exponents.

\begin{lemma}\label{lem:bool_subdesign}
Let $(\gf{2^m}, \calB)$ be the Boolean $\SQS(2^m)$.
Define
\[
\calB^{(j)} := \left\{ \alpha^k \cdot (B + t)  : k \in \bbZ_{2^m - 1},\; t \in \gf{2^m},\; B = \{ 0, 1, \alpha^j, \alpha^l \} \in \calB \right\},
\]
where $\alpha^l =  \alpha^j+1$.
\begin{enumerate}[(i)]
  \item If $m$ is even and $\{j, l\} = \left\{\frac{2^m - 1}{3}, \frac{2(2^m - 1)}{3} \right\}$, then $(\gf{2^m}, \calB^{(j)})$ is a $2$-$(2^m, 4, 1)$ subdesign of $(\gf{2^m}, \calB)$.
  \item Otherwise, $(\gf{2^m}, \calB^{(j)})$ is a $2$-$(2^m, 4, 3)$ subdesign of $(\gf{2^m}, \calB)$.
\end{enumerate}
\end{lemma}

\begin{proof}
The set of all affine transformations of the form $x \mapsto \alpha^k (x + t)$ for $k \in \bbZ_{2^m - 1}$ and $t \in \gf{2^m}$ forms the affine group $G \cong \AGL(1, 2^m)$ acting on $\gf{2^m}$. This group $G$ is sharply $2$-transitive.  
Therefore, the orbit of any $4$-subset $B \subseteq \gf{2^m}$ under $G$ forms a $2$-$(2^m, 4, \lambda)$ design for some $\lambda$.

To compute $\lambda$, it suffices to determine the set-wise stabilizer $\Stab_G(B)$ of the block
  $B = \{ 0, 1, \alpha^j, \alpha^j+1\}$.
This block $B$ itself is an additive subgroup of $\gf{2^m}$ isomorphic to $\bbZ_2^2$, and its stabilizer in $G$ satisfies $\abs{\Stab_G(B)} = 4$ in the generic case.

If $\{j, l\} = \left\{\frac{2^m - 1}{3}, \frac{2(2^m - 1)}{3} \right\}$, which is only possible when $m$ is even, 
then $\{1, \alpha^j, \alpha^l\}$ forms a multiplicative subgroup of order $3$ in $\gf{2^m}^\ast$. In this exceptional case, we have
\[
\Stab_G(B) \cong (B, +) \rtimes \gen{\alpha^{(2^m - 1)/3}},
\]
and  $\abs{\Stab_G(B)} = 4 \times 3 = 12$.

It follows that the number of blocks in $\calB^{(j)}$ is
\[
b = \frac{\abs{G}}{\abs{\Stab_{G}(B)}}
=
\begin{cases}
\frac{2^m (2^m-1)}{12}, &\text{if } \{j, l\} = \left\{\frac{2^m - 1}{3}, \frac{2(2^m - 1)}{3} \right\}, \\
\frac{2^m (2^m-1)}{4}, & \text{otherwise}. 
\end{cases}
\]
Since a $2$-$(2^m, 4, \lambda)$ design has $b = \lambda \cdot \frac{2^m(2^m-1)}{12}$ blocks, 
we obtain $\lambda = 1$ in the exceptional case and $\lambda = 3$ otherwise.

Finally, observe that the sum of all elements in $\alpha^k \cdot (B + t)$ is $0$. 
Hence, $\calB^{(j)} \subseteq \calB$, and the resulting designs are subdesigns of the Boolean $\SQS(2^m)$.
\end{proof}

By considering the orbits of the block set $\calB$ under the action of the affine group $G \cong \AGL(1, 2^m)$, the following result is obtained.

\begin{corollary}\label{cor:bool_decomposition}
Let $(\gf{2^m}, \calB)$ be the Boolean $\SQS(2^m)$.
\begin{enumerate}[(i)]
  \item If $m$ is odd, then $\calB$ can be partitioned as $\calB = \bigcup_{h=1}^{\frac{v - 2}{6}} \calB_h$,
  where $(\gf{2^m}, \calB_h)$ is a $2$-$(2^m,4,3)$ design for each $1 \le h\le \frac{v - 2}{6}$.
  \item If $m$ is even, then $\calB$ can be partitioned as
  $\calB = \calB_0 \cup \bigcup_{h=1}^{\frac{v - 4}{6}} \calB_h$,
  where $(\gf{2^m}, \calB_h)$ is a $2$-$(2^m,4,3)$ design  for each $1 \le h \le \frac{v - 4}{6}$,  and $(\gf{2^m}, \calB_{0})$ is a $2$-$(2^m,4,1)$ design.
\end{enumerate}
\end{corollary}

\begin{remark}
Barker~\cite{baker1976partitioning} showed that 
the Boolean $\SQS(2^{2s})$ can be partitioned into $2^{2s-1}-1$ distinct $2$-$(2^{2s},4,1)$ designs, each of which is resolvable.  
For even $m=2s$, Corollary~\ref{cor:bool_decomposition}~(ii) can also be derived from Barker's decomposition, although the resulting $2$-designs are not necessarily identical to those obtained from Lemma~\ref{lem:bool_subdesign}.  
It is nontrivial to determine whether Barker's decomposition can be applied to the problem of quasi-uniform pairings.
\end{remark}

Each of the sub-$2$-designs in Lemma~\ref{lem:bool_subdesign} and Corollary~\ref{cor:bool_decomposition} is invariant under the action of the affine group $G \cong \AGL(1,2^m)$, and its block set forms a single $G$-orbit.  
In the following discussion, a $G$-orbit will also be referred to as an \emph{affine orbit}.

\begin{lemma}\label{lem:nested_subdesign}
Let $(\gf{2^m}, \calB)$ be the Boolean $\SQS(2^m)$ and let
$(\gf{2^m}, \calB_h)$ be its $2$-$(2^m,4,3)$ subdesign, where $\calB_h$ is a $G$-orbit. 
Take any block of the form
$B = \{ 0, 1, \alpha^j, \alpha^l \} \in \calB_h$ as a $G$-base block, where $\alpha^l = \alpha^j+1$, and let  
$\widetilde{B} = \{ 0,  1 \mid \alpha^j, \alpha^l \}$ be its corresponding nested block. 
Define
\[
\scrB^{(j)} := \left\{ \alpha^k \cdot (\widetilde{B} + t)  : k \in \bbZ_{2^m - 1},\; t \in \gf{2^m} \right\}.
\]
Then, $(\gf{2^m}, \scrB^{(j)})$ is a completely uniform nested design of $(\gf{2^m}, \calB_h)$.
\end{lemma}

\begin{proof}
The argument is analogous to that of Lemma~\ref{lem:bool_subdesign}.  
By construction, $(\gf{2^m}, \scrB^{(j)})$ is a nested design of $(\gf{2^m}, \calB_h)$.  

The nested base block $\widetilde{B} = \{ 0, 1 \mid \alpha^j, \alpha^l \}$ is invariant under 
$x \mapsto x + 1$ and $x \mapsto x + \alpha^j$, since
\begin{align*}
\widetilde{B} + 1&= \{ 1, 0 \mid \alpha^l, \alpha^j \} = \widetilde{B}, \\
\widetilde{B} + \alpha^j &= \{ \alpha^j, \alpha^l \mid 0, 1 \} = \widetilde{B}.
\end{align*}
Thus the stabilizer of $\widetilde{B}$ is isomorphic to $\bbZ_2^2$, and the $G$-orbit $\scrB^{(j)}$ of $\widetilde{B}$ has size
\[
\abs{\scrB^{(j)}} = \frac{\abs{G}}{4}= 2^{m-2}(2^m-1).
\]
This matches the number of blocks in a $2$-$(2^m,4,3)$ design, confirming that $(\gf{2^m}, \scrB^{(j)})$ is indeed a nested design of $(\gf{2^m}, \calB_h)$. 

For the multiplicity of nested pairs, note that the $G$-orbit of the pair $\{0, 1\}$ in $\widetilde{B}$ 
is precisely the set of all $\binom{2^m}{2}$ pairs, that is, the trivial $2$-$(2^m, 2, 1)$ design. 
Consequently, each pair occurs exactly once in $\scrB^{(j)}$, 
and the resulting nested design is completely uniform.
\end{proof}

\begin{theorem}\label{thm:main_boolean_2_des}
For any integer $m \ge 3$, there exists a completely uniform nested $2$-$(2^m,4,3)$ design.
\end{theorem}
\begin{proof}
This is a direct consequence of Lemma~\ref{lem:nested_subdesign}.
\end{proof}

\begin{remark}
Since a whist tournament of order $2^m$ with $m \ge 3$ is equivalent to a resolvable completely uniform nested $2$-$(2^m,4,3)$ design,  
Moore's work~\cite{Moore1896a,Moore1896b} on whist tournaments already implies Theorem~\ref{thm:main_boolean_2_des}.  
Therefore, this result itself is not new. 
\end{remark}

\begin{theorem}\label{thm:main_boolean}
For any integer $m \ge 3$, there exists a nested $\SQS(2^m)$ derived from the Boolean $\SQS(2^m)$, which is completely uniform when $m$ is odd and completely quasi-uniform when $m$ is even.
\end{theorem}

\begin{proof}
First, decompose the Boolean $\SQS(2^m)$ into $2$-$(2^m,4,3)$ subdesigns, together with an additional $2$-$(2^m,4,1)$ subdesign if and only if $m$ is even, as described in Corollary~\ref{cor:bool_decomposition}.  
For each $2$-$(2^m,4,3)$ subdesign, Lemma~\ref{lem:nested_subdesign} ensures the existence of a corresponding completely uniform nested $2$-$(2^m,4,3)$ design.  
The proof is then completed by applying Propositions~\ref{prop:split_2designs_2} and~\ref{prop:split_2designs_4}.
\end{proof}

To clarify the constructions of Lemma~\ref{lem:nested_subdesign} and Theorem~\ref{thm:main_boolean}, three small examples are presented below.

\begin{example}\label{ex:8_bool}
Let $m=3$ and consider the finite field $\gf{8}$ with primitive element $\alpha$, defined by the primitive polynomial $f(x) = x^3 + x + 1 \in \gf{2}[x]$.
In this case, the Boolean $\SQS(8)$ is also a $2$-$(8, 4, 3)$ design, consisting of a single affine orbit. 
Take $B= \{0, 1, \alpha, \alpha^3\}$ as a base block, where $\alpha^3 = \alpha + 1$. 
Following Lemma~\ref{lem:nested_subdesign}, define the base nested block $\widetilde{B}= \{0, 1 \mid \alpha, \alpha^3\}$
and set
\[
\scrB := \left\{ \alpha^k \cdot (\widetilde{B} + t)  : k \in \bbZ_{7},\; t \in \gf{8} \right\}.
\]
This can be partitioned as
\[
\scrB := \left\{ \alpha^k \cdot \widetilde{B}   : k \in \bbZ_{7} \right\} 
\cup \left\{ \alpha^k \cdot (\widetilde{B}+\alpha^2)   : k \in \bbZ_{7} \right\}, 
\]
where 
$\widetilde{B}+\alpha^2 = \{\alpha^2, \alpha^6 \mid \alpha^4, \alpha^5\}$.
Here, $\widetilde{B} = \widetilde{B}+t$ holds for $t \in \{ 1, \alpha, \alpha^3\}$, and
 $\widetilde{B}+\alpha^2 = \widetilde{B}+t$  holds for $t \in \{ \alpha^4, \alpha^5, \alpha^6\}$. 
Notably, $\{\widetilde{B}, \widetilde{B}+\alpha^2\}$ forms a parallel class. 
By Lemma~\ref{lem:nested_subdesign}, $(\gf{8}, \scrB)$ is a completely uniform nested $\SQS(8)$ in which each nested pair has multiplicity $1$. 
Its rotational form was given previously in Example~\ref{ex:8}.
\end{example}

\begin{example}\label{ex:16_bool}
Let $m=4$ and consider the finite field $\gf{16}$ with primitive element $\alpha$, defined by the primitive polynomial $f(x) = x^4 + x + 1$.
The Boolean $\SQS(16)$ is also a $2$-$(16,4,7)$ design, which can be decomposed into two $2$-$(16,4,3)$ designs and one $2$-$(16,4,1)$ design.
Take $\widetilde{B}_1 = \{0, 1 \mid \alpha, \alpha^4\}$ and $\widetilde{B}_2 = \{0, 1 \mid \alpha^2, \alpha^8\}$ as base nested blocks.  
Under the action of the affine group, the base blocks corresponding to $\widetilde{B}_1$ and $\widetilde{B}_2$ generate the associated $2$-$(16,4,3)$ designs. 
Constructing $\scrB^{(1)}$ and $\scrB^{(2)}$ according to Lemma~\ref{lem:nested_subdesign} yields
\begin{align*}
\scrB^{(1)} &= \left\{ \alpha^k \cdot \widetilde{B}_i^{(1)}   : k \in \bbZ_{15}, \; i \in \{0,2,3,6\} \right\},\\
\scrB^{(2)} &= \left\{ \alpha^k \cdot \widetilde{B}_i^{(2)}   : k \in \bbZ_{15}, \; i \in \{0,1,3,7\} \right\},
\end{align*}
where $\widetilde{B}_i^{(j)} := \widetilde{B}_j + \alpha^i$ for $j \in \{1,2\}$.  
Explicitly,
\begin{align*}
\widetilde{B}_0^{(1)} &= \{0, 1 \mid \alpha, \alpha^4 \}, \\
\widetilde{B}_2^{(1)} &=\{ \alpha^2, \alpha^8 \mid \alpha^5, \alpha^{10} \}, \\
\widetilde{B}_3^{(1)} &=\{ \alpha^3, \alpha^{14} \mid \alpha^9, \alpha^7 \}, \\
\widetilde{B}_6^{(1)} &=\{ \alpha^6, \alpha^{13} \mid \alpha^{11}, \alpha^{12} \}, \\[4pt]
\widetilde{B}_0^{(2)} &= \{0, 1 \mid \alpha^2, \alpha^8 \}, \\
\widetilde{B}_1^{(2)} &=\{ \alpha, \alpha^4 \mid \alpha^5, \alpha^{10} \}, \\
\widetilde{B}_3^{(2)} &=\{ \alpha^3, \alpha^{14} \mid \alpha^6, \alpha^{13}\}, \\
\widetilde{B}_7^{(2)} &=\{ \alpha^7, \alpha^{9} \mid \alpha^{12}, \alpha^{11} \}.
\end{align*}
Both $\left\{\widetilde{B}_i^{(1)}\right\}$ and  $\left\{\widetilde{B}_i^{(2)}\right\}$ form parallel classes.  
For the base block $B_5 = \{0, 1, \alpha^5, \alpha^{10}\}$, the affine group action generates the associated $2$-$(16,4,1)$ design.  
Its nested pairs can be partitioned arbitrarily, and the resulting nested design is denoted by $\scrB^{(5)}$.
By Lemma~\ref{lem:nested_subdesign}, $(\gf{16}, \scrB^{(1)} \cup \scrB^{(2)} \cup \scrB^{(5)})$ is a completely quasi-uniform nested $\SQS(16)$ in which each nested pair has multiplicity $2$ or $3$.  
A rotational form of this nested $\SQS(16)$ (with the $2$-$(16,4,1)$ subdesign arranged in a symmetric manner) was given previously in Example~\ref{ex:16}.
\end{example}

\begin{example}\label{ex:32_bool}
Let $m=5$ and consider the finite field $\gf{32}$ with primitive element $\alpha$, defined by the primitive polynomial $f(x) = x^5 + x^2 + 1$.
The Boolean $\SQS(32)$ is also a $2$-$(32, 4, 15)$ design, which can be decomposed into five $2$-$(32, 4, 3)$ designs.
For each sub-2-design, the base nested blocks can be chosen as follows: 
\begin{align*}
\widetilde{B}_1 &=\{ 0, 1 \mid \alpha, \alpha^{18} \}, \\
\widetilde{B}_2 &=\{ 0, 1 \mid \alpha^2, \alpha^{5} \}, \\
\widetilde{B}_4 &=\{ 0, 1 \mid \alpha^4, \alpha^{10} \}, \\
\widetilde{B}_8 &=\{ 0, 1 \mid \alpha^8, \alpha^{20} \}, \\
\widetilde{B}_{16} &= \{ 0, 1 \mid \alpha^{16}, \alpha^9 \}.
\end{align*}
Notably, for each $0 \le i \le 3$, every element of $\widetilde{B}_{2^{i+1}}$ is the square of the corresponding element of $\widetilde{B}_{2^i}$.  
This reflects the $2$-cyclotomic class structure of $\gf{32}$, namely the orbits of $\gf{32}^\ast$ under the Frobenius map $x \mapsto x^2$.
By Lemma~\ref{lem:nested_subdesign}, using these five base nested blocks yields a completely uniform nested $\SQS(32)$ in which each nested pair has multiplicity $5$.  
\end{example}

\section{Applications of completely uniform nested $2$-$(2^m,4,3)$ designs to fractional repetition codes with zero skip
cost }
\label{sect:applications}

This section shows that a completely uniform nested $2$-$(2^m,4,3)$ design yields a fractional repetition (FR) code with locality $2$ and skip cost $0$ for $v=2^m$ packets. This gives an FR code using significantly fewer storage nodes than SQS-based constructions of the same order (see Theorem~\ref{thm:main_FR_2_des}).

FR codes were introduced by El Rouayheb and Ramchandran~\cite{el2010fractional} 
as a combinatorial tool for distributed storage systems.  
They are used as the inner component of DRESS (Distributed Replication based Exact Simple Storage) codes~\cite{pawar2011dress}, 
which combine an outer MDS code with an inner FR code.  
In an FR code, symbols are replicated across storage nodes in a structured manner.  
A failed node can be repaired by transferring symbols from other surviving nodes, called helper nodes, without additional computation.  
This structure enables simple and efficient repair in large-scale storage systems, which also balances the I/O and bandwidth load among helper nodes. 

The formal definition is given below, with notation adjusted from storage coding theory to align with standard usage in combinatorial design theory.

\begin{definition}[FR codes~\cite{el2010fractional}]
A $(b,k,r)$ \emph{fractional repetition} (FR) code $\calC^{\mathrm{FR}}$ is a collection of $b$ subsets 
$B_1, B_2, \ldots, B_b$ of $V:=\{1,2,\ldots,v\}$, 
with $bk = vr$, such that $\abs{B_i} = k$ for each $1 \le i \le b$, 
and each symbol of $V$ belongs to exactly $r$ subsets in $\calC^{\mathrm{FR}}$.  
\end{definition}

In distributed storage systems based on DRESS codes, a file $\vct{x}$ is encoded into $v$ packets $\vct{x}_1, \vct{x}_2, \ldots, \vct{x}_v$ and stored across $b$ storage nodes according to a $(b,k,r)$ FR code $\calC^{\mathrm{FR}}$.  
Node $i$ stores the $k$ packets indexed by the elements of $B_i \in \calC^{\mathrm{FR}}$, that is, $\{ \vct{x}_j : j \in B_i\}$. 
An FR code is said to have \emph{locality} $d$ if any failed node can be repaired by contacting exactly $d$ helper nodes.

In real systems, data access time is characterized not only by the amount of data read (known as I/O costs), but also by the number of contiguous sections accessed on a disk~\cite{wu2021achievable}.  
To capture this factor, Chee et al.~\cite{chee2024repairing} proposed a new metric, called \emph{skip cost}.  
When skip cost is considered, the order of packets within each node becomes crucial.  
Thus, in the following, each node is represented by an ordered $k$-tuple of packet indices.  
The underlying set system (the FR code) is unchanged, where only the order of packets is fixed to account for skip cost.

\begin{definition}[Skip cost~\cite{chee2024repairing}]
Consider a helper node containing packets indexed by $B = (b_1, b_2, \ldots, b_k)$, 
and suppose that $t$ packets are read, say
\[
R = \left\{b_{i_1}, b_{i_2}, \ldots, b_{i_t}\right\}, \qquad i_1 < i_2 < \cdots < i_t.
\]
The \emph{skip cost} of reading $R$ with respect to $B$ is defined as
\[
\operatorname{cost}(R,B) := i_t - i_1 - (t-1).
\]
If the indices in $R$ are consecutive, then the skip cost equals zero.

When multiple helper nodes are used, the total skip cost is the sum of the costs computed in each helper node.
\end{definition}

In terms of combinatorics, an FR code is a $k$-uniform, $r$-regular hypergraph (set system) on $V=\{1,2,\ldots,v\}$ with $b$ edges (blocks).  
In particular, any $2$-$(v,k,\lambda)$ design is a $(b,k,r)$ FR code, where $r = \lambda \frac{v-1}{k-1}$ is the replication number and $b = \lambda \frac{v(v-1)}{k(k-1)}$ is the number of blocks. 

The key property used below is complete uniformity for nested designs, which guarantees, for any failed node $(b_1,b_2,b_3,b_4)$, the existence of two helper nodes (nested blocks) that contain $\{b_2,b_3\}$ and $\{b_1,b_4\}$ as nested pairs (which corresponds to consecutive packets), respectively. 

\begin{figure}[t!]
\centering
\footnotesize
\begin{tabular}{|c|c|c|c|c|c|c|c|c|c|c|c|c|c|}
\hline
$B_1$ & $B_2$ & $B_3$ & $B_4$ & $B_5$ & $B_6$ & $B_7$ &
$B_8$ & $B_9$ & $B_{10}$ & $B_{11}$ & $B_{12}$ & $B_{13}$ & $B_{14}$ \\
\hline
\bbox{\infty} & $\infty$ & $\infty$ & \bbox{\infty} & $\infty$ & $\infty$ & $\infty$ &
$2$ & $3$ & $4$ & $5$ & $6$ & $0$ & $1$ \\
\bbox{0} & $1$ & $2$ & \bbox{3} & $4$ & $5$ & $6$ &
$6$ & $0$ & $1$ & $2$ & $3$ & $4$ & $5$ \\
\bbox{1} & $2$ & $3$ & $4$ & $5$ & $6$ & $0$ &
$4$ & $5$ & $6$ & \bbox{0} & $1$ & $2$ & $3$ \\
\bbox{3} & $4$ & $5$ & $6$ & $0$ & $1$ & $2$ &
$5$ & $6$ & $0$ & \bbox{1} & $2$ & $3$ & $4$ \\
\hline
\end{tabular}
\caption{A $(14, 4, 7)$ FR code with locality $2$ and skip cost $0$ constructed from the completely uniform nested $\SQS(8)$ in Example~\ref{ex:8}.}
\label{fig:fr_sqs8_good}
\end{figure}

An FR code can be represented in array form, where each block $B_i$ is placed in a column corresponding to a storage node.  
Fig.~\ref{fig:fr_sqs8_good} shows a $(14,4,7)$ FR code constructed using the completely uniform nested $\SQS(8)$ given in Example~\ref{ex:8}, in which each nested pair is placed consecutively. 
 (The original labeling $\{\infty, 0, 1, \ldots, 6\}$ is used instead of $\{1,2, \ldots, 8\}$ for clarity.)
This code has locality $2$ with zero skip cost, that is, for any single failed node, two helper nodes are sufficient such that all reads are contiguous.
For instance, when $B_1$ containing $\vct{x}_{\infty}, \vct{x}_0, \vct{x}_1, \vct{x}_3$ fails, $B_4$ and $B_{11}$ can serve as helper nodes such that the skip cost is zero (the corresponding packets are highlighted by boxes in Fig.~\ref{fig:fr_sqs8_good}).

\begin{figure}[t!]
\centering
\footnotesize
\begin{tabular}{|c|c|c|c|c|c|c|c|c|c|c|c|c|c|}
\hline
$B_1$ & $B_2$ & $B_3$ & $B_4$ & $B_5$ & $B_6$ & $B_7$ &
$B_8$ & $B_9$ & $B_{10}$ & $B_{11}$ & $B_{12}$ & $B_{13}$ & $B_{14}$ \\
\hline
$1$ & $1$ & $1$ & $1$ & $1$ & $1$ & $1$ & $5$ & $3$ & $3$ & \bbox{2} & $7$ & $3$ & $3$ \\
\bbox{2} & $2$ & $2$ & $3$ & \bbox{6} & $4$ & $4$ & $6$ & $4$ & $4$ & \bbox{4} & $4$ & $2$ & $2$ \\
$3$ & $5$ & $7$ & $5$ & $3$ & $5$ & $7$ & $7$ & $7$ & $5$ & \bbox{6} & $5$ & $7$ & $5$ \\
\bbox{4} & $6$ & $8$ & $7$ & \bbox{8} & $8$ & $6$ & $8$ & $8$ & $6$ & \bbox{8} & $2$ & $6$ & $8$ \\
\hline
\end{tabular}
\caption{A $(14, 4, 7)$ FR code with locality $2$ and skip cost $2$ constructed from $\SQS(8)$
 (reproduced from \cite[Fig.~2(a)]{chee2024repairing}).}
\label{fig:fr_sqs8_bad}
\end{figure}

For comparison, Fig.~\ref{fig:fr_sqs8_bad} shows an FR code with the same parameters, also constructed from $\SQS(8)$, but with skip cost $2$. For example, suppose the node $B_{11}$ containing $\vct{x}_2, \vct{x}_4, \vct{x}_6, \vct{x}_8$ fails. The packets $\vct{x}_2, \vct{x}_4$ in $B_1$ and $\vct{x}_6, \vct{x}_8$ in $B_5$ are accessed for repair. Since $\vct{x}_2$ and $\vct{x}_4$ in $B_1$ are not contiguous, the gap contributes a skip cost of one. In total, a skip cost of two is required to recover $\vct{x}_2, \vct{x}_4, \vct{x}_6, \vct{x}_8$ from the two helper nodes $B_1$ and $B_5$.

In general, a completely uniform nested $\SQS(v)$ yields a $(b_{\text{SQS}},4,r_{\text{SQS}})$ FR code with locality $2$ and skip cost $0$,  
where $r_{\text{SQS}} = \frac{(v-1)(v-2)}{6}$, $b_{\text{SQS}} = \frac{v(v-1)(v-2)}{24}$, and $v \ge 8$.

However, the $3$-design requirement for SQSs is stronger than necessary for locality $2$ and skip cost $0$. A completely uniform nested $2$-$(v,4,3)$ design already suffices and yields $r_{2\text{-}(v,4,3)}=v-1$ and $b_{2\text{-}(v,4,3)}=\frac{v(v-1)}{4}$. Compared to the SQS-based FR codes, the node count ratio is
\begin{equation*}
\frac{b_{2\text{-}(v,4,3)}}{b_{\text{SQS}}}
=\frac{{v(v-1)}/{4}}{{v(v-1)(v-2)}/{24}}
=\frac{6}{v-2},
\end{equation*}
which is strictly smaller than $1$ for $v>8$ (and equal to $1$ at $v=8$). Here, $v$ denotes the number of packets.

By Theorem~\ref{thm:main_boolean_2_des}, which establishes the existence of a completely uniform nested $2$-$(2^m,4,3)$ design for any $m \ge 3$, 
the following consequence is immediate. 

\begin{theorem}\label{thm:main_FR_2_des}
For any integer $m \ge 3$, there exists a $(b,4,r)$ FR code for $v=2^m$ packets with locality $2$ and skip cost $0$, where  
\[
b = \frac{v(v-1)}{4} = 2^{m-2} (2^m-1), \qquad r = v-1=2^m-1.
\]
\end{theorem}
\begin{proof}
Complete uniformity guarantees, for any failed block $(b_1,b_2,b_3,b_4)$, two helper blocks that contain the pairs $\{b_2,b_3\}$ and $\{b_1,b_4\}$ as consecutive positions, yielding locality $2$ and skip cost $0$. Substituting $\lambda=3$ and $k=4$ into $r=\lambda\frac{v-1}{k-1}$ and $b=\lambda\frac{v(v-1)}{k(k-1)}$ gives $r=v-1$ and $b=\frac{v(v-1)}{4}$.
\end{proof}

Therefore, generalizing completely uniform nested SQSs to $2$-$(v,4,\lambda)$ designs with $\lambda \equiv 0 \pmod{3}$ provides FR codes with locality $2$ and zero skip cost while significantly reducing the number of storage nodes for $v>8$.

\section{More results for small non-Boolean orders}
\label{sect:more_res}

This section provides concrete examples of completely (quasi-)uniform nested $\SQS(v)$ for small non-Boolean values of $v$, most of which were previously unknown, although a few appeared in earlier works under different names.

For $v=10$, the smallest nontrivial order admitting a completely quasi-uniform nested $\SQS(v)$, 
the $\SQS(10)$ is unique up to isomorphism and can be represented as a cyclic SQS on $\bbZ_{10}$ with base blocks 
$\{0,1,5,9\}$, $\{0,2,5,8\}$, and $\{0,1,3,4\}$.  
However, a completely quasi-uniform nested $\SQS(10)$ cannot be obtained simply by partitioning the base blocks and applying cyclic translations.  
For nested pairing, it is more convenient to consider the point set $\bbZ_3 \times \bbZ_3 \cup \{\infty\}$.

\begin{table}[thbp]
\centering
\small
\caption{Base blocks of completely quasi-uniform nested $\SQS(10)$ over $\bbZ_3 \times \bbZ_3 \cup \{\infty\}$ and their translation rules.}
\label{tab:SQS10-partitions}
\begin{tabular}{c|l|l}
\hline
No. & Base nested blocks & Translate by \\
\hline
$\widetilde{B}_1$  & $\{\infty,(0,0) \mid (1,0),(2,0)\}$ & $(\ast,\bbZ_3)$ \\
$\widetilde{B}_2$  & $\{\infty,(0,0) \mid (0,1),(0,2)\}$ & $(\bbZ_3,\ast)$ \\
$\widetilde{B}_3$  & $\{\infty,(0,1) \mid (2,0),(1,2)\}$ & $(\bbZ_3,\ast)$ \\
$\widetilde{B}_4$  & $\{\infty,(0,2) \mid (2,0),(1,1)\}$ & $(\bbZ_3,\ast)$ \\
$\widetilde{B}_5$  & $\{(0,0),(2,2) \mid (1,0),(1,2)\}$ & $(\ast,\bbZ_3)$ \\
$\widetilde{B}_6$  & $\{(0,0),(0,1) \mid (1,0),(2,1)\}$ & $(\ast,\bbZ_3)$ \\
$\widetilde{B}_7$  & $\{(0,0),(1,1) \mid (2,0),(2,1)\}$ & $(\ast,\bbZ_3)$ \\
$\widetilde{B}_8$  & $\{(0,0),(1,2) \mid (0,1),(1,1)\}$ & $(\bbZ_3,\ast)$ \\
$\widetilde{B}_9$  & $\{(0,0),(2,1) \mid (0,2),(2,2)\}$ & $(\bbZ_3,\ast)$ \\
$\widetilde{B}_{10}$ & $\{(0,0),(2,0) \mid (0,1),(2,2)\}$ & $(\bbZ_3,\ast)$ \\
\hline
\end{tabular}
\end{table}

\begin{example}[Completely quasi-uniform nested $\SQS(10)$]\label{ex:cqu-SQS10}
Let $V=\bbZ_3 \times \bbZ_3  \cup  \{\infty\}$.  
For each base nested block $\widetilde{B}_i$ in Table~\ref{tab:SQS10-partitions}, define a semi-cyclic orbit
\[
\scrO_i :=\left\{ \widetilde{B}_i + (a, b) : (a, b) \in T_i \right\},
\]
where $T_i = \{(0,t) : t\in\bbZ_3\}$ if the ``Translate by'' column is $(\ast,\bbZ_3)$, and 
$T_i=\{(t,0) : t\in\bbZ_3\}$ if it is $(\bbZ_3,\ast)$. 
Here, $\infty + (a, b) := \infty$ for all $(a, b) \in \bbZ_3 \times \bbZ_3$. 
For instance, for $\widetilde{B}_1=\{\infty,(0,0)\mid(1,0),(2,0)\}$ with $(\ast,\bbZ_3)$
and $\widetilde{B}_2=\{\infty,(0,0) \mid (0,1),(0,2)\}$ with $(\bbZ_3,\ast)$, the corresponding orbits are respectively  
\begin{align*}
\scrO_1 &=\bigl\{ \{\infty,(0,t)\mid(1,t),(2,t)\} : t\in\bbZ_3 \bigr\},\\
\scrO_2 &=\bigl\{ \{\infty,(t,0) \mid (t,1),(t,2)\} : t\in\bbZ_3 \bigr\}.
\end{align*}
Let $\scrB=\bigcup_{i=1}^{10}\scrO_i$ and then $(V, \scrB)$ is a completely quasi-uniform $\SQS(10)$. 
There are exactly $15$ nested pairs occur with multiplicity $2$ and the remaining $30$ nested pairs occur with multiplicity $1$.
\end{example}

The smallest non-Boolean order admitting a completely uniform nested $\SQS(v)$ is $v=14$.  
There exist four non-isomorphic $\SQS(14)$~\cite{MendelsohnHung1972}.  
The following example is based on the unique $\SQS(14)$ with the largest automorphism group, namely the $\SQS(14)$ whose automorphism group is $\AGL(1,7) \cong \bbZ_7 \rtimes \bbZ_7^\ast$.  

\begin{table}[thbp]
\centering
\small
\caption{Base blocks of completely uniform nested $\SQS(14)$ over $\bbZ_7 \times \{0, 1\}$.}
\label{tab:SQS14-partitions}
\begin{tabular}{c|l}
\hline
No. & Base nested blocks \\
\hline
$\widetilde{B}_{1}$  & $\{(0,0), (1,0) \mid (3,1), (4,1)\}$ \\
$\widetilde{B}_{2}$  & $\{(0,0), (1,0) \mid (4,0), (2,0)\}$ \\
$\widetilde{B}_{3}$  & $\{(0,0), (2,0) \mid (2,1), (0,1)\}$ \\
$\widetilde{B}_{4}$  & $\{(0,0), (4,0) \mid (6,1), (3,1)\}$ \\
$\widetilde{B}_{5}$  & $\{(0,0), (0,1) \mid (4,0), (4,1)\}$ \\
$\widetilde{B}_{6}$  & $\{(0,0), (1,1) \mid (0,1), (1,0)\}$ \\
$\widetilde{B}_{7}$  & $\{(0,0), (1,1) \mid (6,1), (2,0)\}$ \\
$\widetilde{B}_{8}$ & $\{(0,0), (2,1) \mid (4,1), (1,1)\}$ \\
$\widetilde{B}_{9}$  & $\{(0,0), (2,1) \mid (1,0), (5,0)\}$ \\
$\widetilde{B}_{10}$  & $\{(0,0), (3,1) \mid (5,1), (2,0)\}$ \\
$\widetilde{B}_{11}$ & $\{(0,0), (5,1) \mid (2,1), (4,0)\}$ \\
$\widetilde{B}_{12}$ & $\{(0,0), (6,1) \mid (5,1), (1,0)\}$ \\
$\widetilde{B}_{13}$ & $\{(0,1), (6,1) \mid (5,1), (3,1)\}$ \\
\hline
\end{tabular}
\end{table}

\begin{example}[Completely uniform nested $\SQS(14)$]\label{ex:cu-SQS14}
Let $V=\bbZ_7 \times \{0,1\}$.  
For each base nested block $\widetilde{B}_i$ in Table~\ref{tab:SQS14-partitions}, 
define its cyclic orbit by
\[
\scrO_i := \left\{  \widetilde{B}_i + (t,\ast) : t \in \bbZ_7 \right\},
\]
where $(x,y)+(t,\ast):=(x+t,y)$ for any $x, t \in \bbZ_7$. 
Let $\scrB=\bigcup_{i=1}^{13}\scrO_i$.  
Then $(V,\scrB)$ is a completely uniform nested $\SQS(14)$.  
\end{example}

\begin{remark}
In the framework of ``nested BIBDs'' (cf. Definition~\ref{def:preece}), Morgan, Preece, and Rees~\cite{MorganPreeceRees2001} presented three examples of completely uniform nested $\SQS(14)$, listed as No.~57 (labeled Ca1, Ca2, Ca3) in Table~1 of~\cite{MorganPreeceRees2001}.  
Quite recently, Rosin~\cite{rosin2025using} proposed a heuristic method, assisted by Large Language Models (LLMs), for searching combinatorial design structures. Through this approach, a completely uniform nested $\SQS(14)$ was discovered.
\end{remark}

The next example is based on a rotational $\SQS(44)$ constructed by a classical method of Carmichael~\cite{carmichael1937groups} for constructing an $\SQS(q+1)$ using the action of the projective special linear group $\PSL(2,q)$ on the projective line $\gf{q} \cup \{\infty\}$, 
where $q$ is a prime power satisfying $q \equiv 7 \pmod{12}$.  

It is known~\cite{hartman1980resolvable} that Carmichael's $\SQS(q+1)$ contains the stabilizer of $\infty$ in $\PSL(2,q)$ 
as a subgroup of its automorphism group.  
This stabilizer is
\[
\Gamma \cong (\gf{q},+) \rtimes \langle \alpha^2 \rangle,
\]
and consists of the maps $\gamma: x \mapsto \alpha^{2k}x + c$ for $0 \le k < (q-1)/2$ and $c \in \gf{q}$, 
where $\alpha$ is a primitive element of $\gf{q}$.  

\begin{example}[Completely uniform nested $\SQS(44)$]\label{ex:cu-SQS44}
Let $q = 43$. Consider the primitive element $\alpha = 3$. 
By acting $\PSL(2, q)$ on $\{\infty, 0, 1, 37\}$,  
a rotational $\SQS(44)$ is obtained.  
In this case, there are five $\Gamma$-orbits, and the corresponding $\Gamma$-base blocks can be chosen as follows:
\[
\begin{aligned}
B_\infty &= \{\infty, 1, 36, 6\}, \\
B_0 &= \{0, 3, 22, 18\}, \\
B_1 &= \{1, 8, 14, 29\}, \\
B_2 &= \{15, 22, 25, 39\}, \\
B_3 &= \{11, 33, 35, 40\}.
\end{aligned}
\]
Using $B_\infty$ and $B_1$, corresponding rotational base nested  blocks can be constructed as follows: 
\begin{align*}
\scrB_\infty
   &=\Bigl\{
       \{\infty, 1  \mid 36,  6\},
       \{\infty, 4  \mid 15, 24\},
       \{\infty, 11 \mid  9, 23\},\\
   &\qquad
       \{\infty, 16 \mid 17, 10\},
       \{\infty, 21 \mid 25, 40\},
       \{\infty, 35 \mid 13, 38\},
       \{\infty, 41 \mid 14, 31\}
     \Bigr\},
\\
\scrB_0
   &=\Bigl\{
       \{0,  3 \mid 22, 18\},
       \{0, 12 \mid  2, 29\},
       \{0, 33 \mid 27, 26\}, \\
   &\qquad
       \{0,  5 \mid  8, 30\},
       \{0, 20 \mid 32, 34\},
       \{0, 19 \mid 39, 28\},
       \{0, 37 \mid 42,  7\}
     \Bigr\}.
\end{align*}
Here, the first nested block in $\scrB_\infty$ (resp., $\scrB_0$) corresponds to a nested block of $B_\infty$ (resp., $B_0$), say $\widetilde{B}_\infty$ (resp., $\widetilde{B}_0$).
The remaining nested blocks are obtained by multiplying $\widetilde{B}_\infty$ (resp., $\widetilde{B}_0$) by suitable quadratic residues in $\gf{q}$.  
For $i \in \{1, 2, 3\}$, let $\widetilde{B}_i$ denote a nested block of $B_i$, obtained by partitioning $B_i$ into unordered pairs in an arbitrary way. Define 
\begin{align*}
\scrB_i &= \left\{ \alpha^{2k} \widetilde{B}_i : 0 \le k < (q-1)/2 \right\},
\quad i \in \{1, 2, 3\},
\end{align*}
where $\alpha$ is a primitive element of $\gf{q}$.   
Then $\bigcup_{i \in \{\infty,0,1,2,3\}} \scrB_i$ forms the set of rotational base nested blocks.  
Finally, apply translations over $\gf{q}$ to the rotational base nested blocks by defining
\[
\scrB := \bigcup_{i \in \{\infty,0,1,2,3\}} \; \bigcup_{\widetilde{B} \in \scrB_i} 
          \left\{\widetilde{B} + t : t \in \gf{q} \right\}.
\]
Then $(V,\scrB)$ is a completely uniform nested $\SQS(44)$, in which each pair occurs with multiplicity $7$.
\end{example}

In fact, the construction in Example~\ref{ex:cu-SQS44} for $\SQS(44)$ can be generalized to all prime powers $q \equiv 7 \pmod{12}$.  
However, the details are quite involved and require techniques that are not available in the Boolean case.  
As this paper is intended to focus mainly on the Boolean case, a full proof of the generalization is beyond its scope and will be presented in a forthcoming work.

\begin{table}[t]
\centering
\small
\caption{Base blocks of completely uniform nested $\SQS(50)$ over $\bbZ_{49} \cup \{\infty\}$.}
\label{tab:SQS50-partitions}
\begin{tabular}{llll}
\hline
$\{\infty,0 \mid 1,3\}$ & $\{\infty,0 \mid 4,19\}$ & $\{\infty,0 \mid 5,18\}$ & $\{\infty,0 \mid 6,23\}$ \\
$\{\infty,0 \mid 7,21\}$ & $\{\infty,0 \mid 8,24\}$ & $\{\infty,0 \mid 9,20\}$ & $\{\infty,0 \mid 10,22\}$ \\
$\{0,1 \mid 7,8\}$ & $\{0,1 \mid 14,15\}$ & $\{0,1 \mid 28,29\}$ & $\{0,2 \mid 7,9\}$ \\
$\{0,2 \mid 14,16\}$ & $\{0,2 \mid 26,22\}$ & $\{0,2 \mid 28,30\}$ & $\{0,3 \mid 6,12\}$ \\
$\{0,4 \mid 3,2\}$ & $\{0,5 \mid 2,43\}$ & $\{0,6 \mid 39,36\}$ & $\{0,7 \mid 13,37\}$ \\
$\{0,7 \mid 17,36\}$ & $\{0,8 \mid 48,4\}$ & $\{0,9 \mid 42,27\}$ & $\{0,11 \mid 16,24\}$ \\
$\{0,11 \mid 29,34\}$ & $\{0,11 \mid 41,22\}$ & $\{0,11 \mid 45,37\}$ & $\{0,12 \mid 9,43\}$ \\
$\{0,13 \mid 25,15\}$ & $\{0,14 \mid 27,20\}$ & $\{0,15 \mid 11,48\}$ & $\{0,15 \mid 36,10\}$ \\
$\{0,16 \mid 9,35\}$ & $\{0,16 \mid 30,21\}$ & $\{0,17 \mid 34,5\}$ & $\{0,18 \mid 7,19\}$ \\
$\{0,18 \mid 24,9\}$ & $\{0,18 \mid 43,27\}$ & $\{0,20 \mid 41,35\}$ & $\{0,21 \mid 26,12\}$ \\
$\{0,21 \mid 27,13\}$ & $\{0,21 \mid 36,45\}$ & $\{0,22 \mid 7,24\}$ & $\{0,22 \mid 37,12\}$ \\
$\{0,22 \mid 48,18\}$ & $\{0,23 \mid 7,20\}$ & $\{0,23 \mid 17,46\}$ & $\{0,23 \mid 47,22\}$ \\
$\{0,25 \mid 29,48\}$ & $\{0,26 \mid 11,18\}$ & $\{0,27 \mid 32,8\}$ & $\{0,27 \mid 33,3\}$ \\
$\{0,27 \mid 46,22\}$ & $\{0,28 \mid 14,21\}$ & $\{0,28 \mid 18,12\}$ & $\{0,28 \mid 45,10\}$ \\
$\{0,29 \mid 9,26\}$ & $\{0,29 \mid 17,24\}$ & $\{0,29 \mid 19,2\}$ & $\{0,29 \mid 40,23\}$ \\
$\{0,30 \mid 17,39\}$ & $\{0,30 \mid 29,47\}$ & $\{0,30 \mid 36,40\}$ & $\{0,31 \mid 23,11\}$ \\
$\{0,31 \mid 48,33\}$ & $\{0,32 \mid 18,42\}$ & $\{0,32 \mid 28,33\}$ & $\{0,34 \mid 46,15\}$ \\
$\{0,35 \mid 11,4\}$ & $\{0,35 \mid 24,45\}$ & $\{0,36 \mid 18,34\}$ & $\{0,36 \mid 28,38\}$ \\
$\{0,36 \mid 33,23\}$ & $\{0,37 \mid 5,29\}$ & $\{0,37 \mid 28,34\}$ & $\{0,38 \mid 25,16\}$ \\
$\{0,39 \mid 1,41\}$ & $\{0,39 \mid 2,38\}$ & $\{0,39 \mid 8,34\}$ & $\{0,39 \mid 10,23\}$ \\
$\{0,39 \mid 14,40\}$ & $\{0,40 \mid 32,21\}$ & $\{0,41 \mid 8,46\}$ & $\{0,41 \mid 21,37\}$ \\
$\{0,41 \mid 27,7\}$ & $\{0,42 \mid 23,37\}$ & $\{0,43 \mid 14,45\}$ & $\{0,43 \mid 32,41\}$ \\
$\{0,43 \mid 48,39\}$ & $\{0,44 \mid 6,27\}$ & $\{0,44 \mid 14,41\}$ & $\{0,44 \mid 39,31\}$ \\
$\{0,44 \mid 43,47\}$ & $\{0,45 \mid 22,38\}$ & $\{0,45 \mid 30,18\}$ & $\{0,45 \mid 32,16\}$ \\
$\{0,46 \mid 7,4\}$ & $\{0,46 \mid 14,11\}$ & $\{0,46 \mid 28,25\}$ & $\{0,48 \mid 32,36\}$ \\
\hline
\end{tabular}
\end{table}

\begin{example}[Completely uniform nested $\SQS(50)$]\label{ex:cu-SQS50}
Let $V = \bbZ_{49} \cup \{\infty\}$.  
Define $\scrB$ by translating each of the $100$ base nested blocks listed in Table~\ref{tab:SQS50-partitions} 
under the action of the cyclic group $\bbZ_{49}$, fixing $\infty$.  
Then $(V,\scrB)$ is a completely uniform nested $\SQS(50)$.  

The underlying $\SQS(50)$ is rotational, obtained by the recursive construction of Ji and Zhu~\cite{ji2002improved}, 
with the rotational $\SQS(8)$ having base blocks $\{\infty,0,1,3\}$ and $\{2,6,4,5\}$ as the ingredient design (cf.~Example~\ref{ex:8}).  
The pair partitions required for the nested blocks were obtained by computer search.
\end{example}

\begin{table}[t]
\centering
\small
\caption{Existence of completely uniform nested $\SQS(v)$ for $8 \le v \leq 50$}\label{tab:cu_SQS_le_50}
\begin{tabular}{c|l|l}
\hline
$v$ & Existence and references & Remarks \\
\hline
 8  & Exists (Theorem~\ref{thm:main_boolean}, Example~\ref{ex:8_bool}) & Boolean, rotational \\
14  & Exists (Example~\ref{ex:cu-SQS14}) &  Semi-cyclic \\
20  & Exists (\cite[Example~4.6]{chee2024pairs}) & Rotational \\
26  & Exists (\cite[Example~4.7]{chee2024pairs}) & Rotational \\
32  & Exists (Theorem~\ref{thm:main_boolean}, Example~\ref{ex:32_bool}) & Boolean, rotational \\
38  & Exists (\cite[Example~4.8]{chee2024pairs}) & Rotational \\
44  & Exists (Example~\ref{ex:cu-SQS44}) & Rotational \\
50  & Exists (Example~\ref{ex:cu-SQS50}) & Rotational \\
\hline
\end{tabular}
\end{table}

Table~\ref{tab:cu_SQS_le_50} summarizes the existence of completely uniform nested $\SQS(v)$ for $8 \le v \le 50$ with $v \equiv 2 \pmod{6}$, all of which have been established. Most results in this range arise from rotational SQSs, including the Boolean cases. 
The next unresolved case is $v=56$.  
Because the existence of a rotational $\SQS(56)$ is undetermined, this case is especially challenging and may be difficult to approach even with computer search.  

In contrast, for completely quasi-uniform nested $\SQS(v)$ with $v \equiv 4 \pmod{6}$, 
the situation remains largely unresolved.  
Apart from the Boolean cases and the smallest nontrivial case $\SQS(10)$ given in Example~\ref{ex:cqu-SQS10}, no further examples are currently known. 

\section{Concluding remarks}
\label{sect:conclude}

This study proposes explicit constructions for completely (quasi-)uniform nested SQSs based on Boolean $\SQS(2^m)$, resolving two open problems from~\cite{chee2024pairs} by constructing infinite families of such designs for odd $m \ge 3$ (completely uniform) and even $m \ge 4$ (completely quasi-uniform). 

The constructions presented in this study rely heavily on the rich algebraic and geometric structure of the Boolean SQSs, 
in particular their affine invariance ($2$-transitivity), rotational automorphisms, and resolvable properties.  

From the perspective of transitivity, it is well known that the automorphism group of the Boolean $\SQS(2^m)$ 
(or equivalently, the affine geometry $\AG(m,2)$) is the affine general linear group 
$\AGL(m,2) \simeq \gf{2^m} \rtimes \GL(m,2)$, which is $3$-transitive (see \cite[Section~5.2]{lindner1978steiner}).  
However, 
in our construction, the subgroup $\AGL(1,2^m)$, which is sharply 2-transitive, plays the key role.  
In general, the full automorphism group of an SQS cannot be applied directly to its nested design, 
as it may fail to preserve the nested pairs within each nested block.  This suggests that an appropriate subgroup of the full automorphism group is often essential for constructing completely uniform nested pairings.  

A potential direction for future research is to generalize the ideas used in Sections~\ref{sect:bool_construction} and \ref{sect:more_res} to broader classes of SQSs, 
such as those with large automorphism groups or resolvable rotational SQSs.  
Related studies on $2$-designs, such as cyclically resolvable rotational $2$-designs~\cite{anderson1993cyclically,jimbo1984recursive}, cyclically resolvable cyclic $2$-designs~\cite{genma1997cyclic,lam1999cyclically,mishima1997some,mishima2000recursive,sarmiento2000resolutions}, and resolvable difference families~\cite{buratti1997resolvable}, received considerable attention in design theory from the 1980s to the 2000s, though the general case remains far from settled even for $2$-designs.  
However, related study for $3$-designs has been less developed, with few infinite families known. 
Sawa~\cite{sawa2007cyclic} explored cyclically resolvable cyclic $3$-$(v, 4, \lambda)$ designs with $\lambda \equiv 0 \pmod{3}$ and proved their existence for all admissible $v$, but the case of $\lambda = 1$ (Steiner quadruple systems) remains quite challenging.  
The resolvable rotational case for $3$-designs is largely unexplored. 
The open problems concerning nested SQSs motivate further investigation into the existence and construction of resolvable $3$-designs admitting cyclic, rotational, or, more generally, point-regular automorphism groups.

Moreover, as discussed in Section~\ref{sect:applications}, 
completely uniform nested $2$-$(v,4,3)$ designs are also of independent interest from the perspective of applications.  
Their existence and construction problems appear closely related to the study of cyclically resolvable rotational $2$-designs, 
which is itself an interesting topic in combinatorial design theory.

\section*{Acknowledgements}
The authors would like to thank the anonymous reviewers for their constructive comments, 
which clarified the significance of this study and improved the overall structure of the paper.  

This study was partly supported by JSPS KAKENHI Grant Nos.~22K13949, 22K11943, and 24K14819.

\bibliographystyle{abbrv}
\bibliography{bibBSQS}

\end{document}